\documentclass{article}

\usepackage[hyperref]{package_RFb}

\usepackage[numbers]{natbib}
\usepackage[top=3cm,bottom=3cm,left=3.5cm,right=3.5cm]{geometry}

\title{Correction to ``The stepping stone model in a random environment'' - limit theorems for the occupation time and intersection time of reversible random walks in random environments}
\author{Rapha\"el Forien\footnote{BioSP, INRAE, France. 
		\url{raphael.forien@inrae.fr}, {0000-0002-8901-4921}.}}

\DeclareDocumentCommand{\mean}{m}{\int_\Omega #1 d \mu}

\DeclareDocumentCommand\Pq{o m g}{% Probabilite quenched
	\mathbb{P}^\omega%
	\IfNoValueF{#1}{_{#1}}% Indice optionnel
	\left(%
	\IfNoValueF{#3}{\left.}%
	#2%
	\IfNoValueF{#3}{\:\right|\: #3}% Conditionnement optionnel
	\right)}
\DeclareDocumentCommand\Eq{o m g}{% Esperance \E[indice]{variable}{condition}
	\mathbb{E}^\omega%
	\IfNoValueF{#1}{_{#1}}% Indice optionnel
	\left[%
	\IfNoValueF{#3}{\left.}%
	#2%
	\IfNoValueF{#3}{\:\right|\: #3}% Conditionnement optionnel
	\right]}

\DeclareDocumentCommand{\plusminusone}{}{\lbrace -1, 1 \rbrace}

\begin{document}
	\maketitle
	
	\begin{abstract}
		This note corrects a mistake in the author's above-mentionned published work regarding the intersection local time of two independent random walks in a random environment. The random walks each behave as nearest neighbour random walks in random conductances in the same random environment. The proof of the main result in the original work relied on the derivation of the scaling limit of an additive functionnal of the environment at the locations where the two independent random walks meet, which was shown to be given by a constant times the intersection local time of two Brownian motions. This note shows how the value of the constant denoted by $ \gamma $ in this work should be updated to correct the statement. The derivation of this constant uses new results on the scaling limit of the occupation time of reversible random walks in a random environment.
		
		~~
		
		\textbf{Subject: } 60F17; 60J55; 60G99.
	\end{abstract}

	\section*{Introduction}
	
	In \cite{forien_stepping_2019}, the author studied a model of genetic isolation by distance in a population divided in discrete demes with the graph structure induced by $ \Z $.
	The size of the subpopulation in each deme was given by an ergodic stationary random field on $ \Z $, and allele frequencies in each deme was assumed to follow a Wright-Fisher diffusion with exchanges between neighbouring demes.
	It was shown that, under a diffusive scaling of time and space, the dynamics of allele frequencies converges to the solution to a stochastic partial differential equation also called continuous-sites stepping stone model, with an effective diffusion coefficient and an effective strength of genetic drift expressed as functions of the law of the random field.
	The proof relied on an analysis of the moment dual of the allele frequencies process, which takes the form of a system of coalescing random walks in a random environment (determined by the random field). 
	Under the same diffusive scaling, this system of random walks converges to a system of coalescing Brownian motions, which coalesce when their intersection local time exceeds a random expontial threshold.
	The effective diffusion and strength of genetic drift parameters of the SPDE then correspond to the infinitesimal variance and the parameter of the exponential threshold, respectively.
	
	It turns out that the expression for the effective strength of genetic drift / coalescence parameter given in \cite{forien_stepping_2019} is incorrect.
	The next section details how the statements of \cite{forien_stepping_2019} should be modified to correct this, and the rest of this paper is devoted to the proof of the updated result.
	The correction only concerns the convergence of the system of coalescing random walks in a random environment, so we do not recall here the details of the stepping stone model.
	The derivation of the correct effective coalescence parameter requires new results on the scaling limit of the occupation time of two independent random walks in a joint random environment.
	In passing, we also use the same technique to prove a similar result on the scaling limit of the occupation time of a single random walk in such a random environment.
	
	\section{Correction to the statements} \label{sec:statements}
	
	Let us here detail how the statements in \cite{forien_stepping_2019} need to be modified.
	The main result of the paper, Theorem~1.4, holds with the only modification that, contrary to what is stated in (1.5),
	\begin{equation} \label{gamma}
		\gamma = \frac{\mean{N (N_3)^2}}{\left( \mean{N N_3} \right)^2}.
	\end{equation}
	The other statements are similarly affected, notably Theorem~2.8.
	The mistake appears in Section~3.3 (see below), and no correction is needed outside of this section.
	Incidentally, the corrected proof that we give below makes Section~3.2 redundant, even though the results in it (Proposition~3.3) remain correct.
	As a result, the correction given below should be taken as a replacement for Sections~3.2 and 3.3 in \cite{forien_stepping_2019}.
	
	Point 2 of Remark~1.5 in \cite{forien_stepping_2019} then needs to be modified as well.
	Indeed, by the Cauchy-Schwarz inequality,
	\begin{equation*}
		\left( \mean{N N_3} \right)^2 \leq \mean{N (N_3)^2} \mean{N}.
	\end{equation*}
	Hence
	\begin{equation} \label{ineq_gamma_N}
		\frac{1}{\gamma} \leq \mean{N}.
	\end{equation}
	Note that this inquality is the opposite of the one obtained in Remark~1.5 of \cite{forien_stepping_2019} (under an additional assumption).
	This means that the ``apparent'' population density $ 1/\gamma $ which would be estimated by fitting a homogeneous model of isolation by distance is in fact \emph{smaller} than the average population density $ \mean{N} $.
	%	This turns out to be independent of the fact that lineages spend more time in crowded regions (see Subsection~\ref{subsec:general_rates}).
	
	\section{Model and notations} \label{sec:model}
	
	\subsection{Random walk in a random environment} \label{subsec:def_xi}
	
	Let us here recall the setting of Section~3 in \cite{forien_stepping_2019}.
	Let $ (\Omega, \mathcal{B}, \mu) $ be a probability space on which is defined a family of measurable maps $ (T^x)_{x\in \Z} $, $ T^x : \Omega \to \Omega $ such that
	\begin{enumerate}[i)]
		\item $ T^x \circ T^y = T^{x+y} $ for all $ x, y \in \Z $ and $ T^0 = Id_{\Omega} $, \label{ass:group}
		\item $ \mu(T^x A) = \mu(A) $ for all $ A \in \mathcal{B} $ and $ x \in \Z $, \label{ass:translation_invariance}
		\item if $ A \in \mathcal{B} $ is such that $ T^x A = A $ for all $ x \in \Z $, then $ \mu(A) \in \{0, 1 \} $.\label{ass:ergodicity}
	\end{enumerate}
	
	Now let $ j : \Omega \times \plusminusone \to \R_+ $ be a measurable function and assume that there exists a measurable function $ \pi : \Omega \to \R_+ $ such that, for $ \mu $-almost every $ \omega \in \Omega $ and $ z \in \lbrace -1, 1 \rbrace $,
	\begin{equation} \label{reversibility}
		\pi(\omega) j(\omega,z) = \pi(T^z \omega) j(T^z\omega, -z),
	\end{equation}
	and
	\begin{equation*}
		\int_\Omega \pi(\omega) \mu(d\omega) = 1.
	\end{equation*}
	Further assume that there exists $ K > 0 $ such that, for $ z \in \plusminusone $ and $ \mu $-almost every $ \omega \in \Omega $,
	\begin{align} \tag{$\mathcal{U.E.}$} \label{uniform_ellipticity}
		\frac{1}{K} \leq j(\omega,z) \leq K, && \frac{1}{K} \leq \pi(\omega) \leq K.
	\end{align}
	We also consider another measureable function $ N : \Omega \to \R_+ $ and assume that, for $ \mu $-almost every $ \omega \in \Omega $,
	\begin{equation} \label{UE_N}
		\frac{1}{K} \leq N(\omega) \leq K.
	\end{equation}
	Note that, from any measurable $ f : \Omega \to \R $, we can define a random field on $ \Z $, denoted by $ \bm{f} : \omega \to \lbrace f(\omega, x), x \in \Z \rbrace $, through the relation
	\begin{equation}
		f(\omega,x) := f(T^x \omega).
	\end{equation}
	By assumptions~\ref{ass:group}-\ref{ass:ergodicity}, $ \bm{f} $ is stationary and ergodic with respect to translations.
	
	In \cite{forien_stepping_2019}, we use $ \bm{N} $ and $ \bm{j}(\cdot,\pm 1) $ to define a stepping stone model in which the size of the deme located at $ x \in \Z $ is proportional to $ N(\omega,x) $, and ancestral lineages follow random walks with jump rates given by $ j(\omega,x,\pm 1) $ (see Definition~\ref{def:rw} below).
	We refer to \cite{forien_stepping_2019} for details on this stepping stone model, as the present note only deals with its moment dual, defined as a system of coalescing random walks in a random environment.
	
	\begin{definition} \label{def:rw}
		For $ \omega \in \Omega $, let $ (\xi_t, t \geq 0) $ be a random walk on $ \Z $ such that it jumps from $ x \in \Z $ to $ x + z \in \Z $ at rate $ j(\omega,x,z) $ if $ z \in \plusminusone $ and at rate zero otherwise.
		In the following, $ \Pq[x]{\cdot} $ (resp. $ \Eq[x]{\cdot} $) denotes the distribution (resp. the expectation with respect to the distribution) of the random walk $ (\xi_t, t \geq 0) $ started from $ x \in \Z $ conditionally on the environment $ \omega $.
	\end{definition}
	
	We remark that \cite{forien_stepping_2019} assumes a particular dependence between $ j(\omega,z) $ and $ N(\omega) $, so that \eqref{uniform_ellipticity} follows from \eqref{UE_N}, but the proofs readily extend to the present setting, provided \eqref{uniform_ellipticity} is satisfied.

	\subsection{Functional and local central limit theorems for the random walk} \label{subsec:clt}
	
	Note that, in view of \eqref{reversibility}, $ \bm{\pi} $ defines a reversible measure for the random walk $ (\xi_t, t \geq 0) $ on $ \Z $.
	Together with \eqref{uniform_ellipticity}, this implies the following functional central limit theorem for the random walk, obtained in \cite{forien_stepping_2019} adapting the arguments of \cite{lam_quenched_2014,derrien_local_2015,depauw_variance_2009}.
	For $ n \in \N $, let $ (\xi^n_t, t \geq 0) $ be defined as
	\begin{equation*}
		\xi^n_t := \frac{1}{\sqrt{n}} \xi_{nt}.
	\end{equation*}
	For $ T > 0 $, let $ D([0,T], \R) $ denote the space of \cadlag real-valued functions endowed with the usual Skorokhod topology.
	
	\begin{theorem}[Theorem~2.3 in \cite{forien_stepping_2019}] \label{thm:functional_clt}
		Fix $ T > 0 $ and assume that $ \xi^n_0 $ is deterministic and converges to $ x_0 \in \R $.
		Suppose that assumptions~\ref{ass:group}-\ref{ass:ergodicity} and \eqref{uniform_ellipticity} are satisfied.
		Then, for $ \mu $-almost every environment $ \omega \in \Omega $, as $ n \to \infty $, $ (\xi^n_t, t \in [0,T]) $ converges in distribution in $ D([0,T], \R) $ to Brownian motion started from $ x_0 $ with variance $ \sigma^2 $ given by
		\begin{equation*}
			\sigma^2 = 2 \left( \int_\Omega \frac{1}{\pi(\omega) j(\omega,1)} \mu(d\omega) \right)^{-1}.
		\end{equation*}
	\end{theorem}
	
	Adapting the arguments of \cite{derrien_local_2015} and using \eqref{uniform_ellipticity}, the author was also able to prove the following local central limit theorem in \cite{forien_stepping_2019}.
	For $ \omega\in \Omega $, $ t \geq 0 $ and $ x, y \in \Z $, set
	\begin{align} \label{def_g_omega}
		g^\omega_t(x,y) := \Pq[x]{\xi_t = y}
	\end{align}
	and for $ t > 0 $, $ x \in \R $, define
	\begin{align*}
		G_t(x) := \frac{1}{\sqrt{2\pi \sigma^2 t}} \exp \left( - \frac{x^2}{2\sigma^2 t} \right).
	\end{align*}
	
	\begin{theorem}[Theorem~2.3 in \cite{forien_stepping_2019}] \label{thm:local_clt}
		For all $ 0 < \varepsilon < T $ and $ R > 0 $ and for $ \mu $-almost every environment $ \omega \in \Omega $,
		\begin{align*}
			\lim_{n \to \infty} \: \sup_{t \in [\varepsilon, T]} \: \max_{\substack{x \in B(0,R\sqrt{n}) \cap \Z \\ y \in B(0, R \sqrt{n}) \cap \Z}} \: \abs{ \sqrt{n} \frac{g^\omega_{nt}(x,y)}{\pi(\omega,y)} - G_t \left( \frac{x-y}{\sqrt{n}} \right) } = 0.
		\end{align*}
	\end{theorem}
	
	The proof of Theorem~\ref{thm:local_clt} relied on the following two lemmas, which were also adapted from \cite{derrien_local_2015}.
	Let us set, for $ t \geq 0 $ and $ x, y \in \Z $,
	\begin{align} \label{def:ht}
		h^\omega_t(x,y) := \frac{g^\omega_t(x,y)}{\pi(\omega,y)}.
	\end{align}
	
	\begin{lemma}[Lemma~A.1 in \cite{forien_stepping_2019}] \label{lemma:bound_h}
		There exists a constant $ C > 0 $ such that, for $ \mu $-almost every $ \omega \in \Omega $,
		\begin{align*}
			\sup_{t \geq 0} \sup_{x, y \in \Z} \sqrt{t} \, h^\omega_t(x,y) \leq C.
		\end{align*}
	\end{lemma}
	
	\begin{lemma}[Lemma~A.2 in \cite{forien_stepping_2019}] \label{lemma:continuity_h}
		There exists a constant $ C > 0 $ such that, for $ \mu $-almost every $ \omega \in \Omega $,
		\begin{align*}
			\sup_{t \geq 0} \sup_{x, y, z \in \Z} t^{3/4} \frac{\abs{h^\omega_t(x,y) - h^\omega_t(x,z)}}{\abs{y-z}^{1/2}} \leq C.
		\end{align*}
	\end{lemma}
	
	\subsection{Intersection local time of random walks in a joint random environment} \label{subsec:intersection}
	
	The problematic argument in \cite{forien_stepping_2019} is in Section~3.3, which deals with a pair of independent random random walks following Definition~\ref{def:rw}.
	By Theorem~\ref{thm:functional_clt} and by the Skorokhod representation theorem, for $ \mu $-almost every $ \omega $, there exists a probability space on which two sequences of processes $ (\xi^{1,n}_t, \xi^{2,n}_t)_{t \in [0,T]} $ and two independent Brownian motions $ (X^1_t, X^2_t)_{t \in [0,T]} $ are defined and such that
	\begin{enumerate}[i)]
		\item for all $ n \geq 1 $, $ (\xi^{i,n}_t)_{t \in [0,T]} $ is distributed like $ (\xi_t)_{t \in [0,T]} $, for $ i \in \{1, 2\} $,
		\item $ \xi^{1,n} $ and $ \xi^{2,n} $ are independent for all $ n \geq 1 $,
		\item the sequence of processes $ (\frac{1}{\sqrt{n}} \xi^{i,n}_{nt})_{t\in [0,T]} $ converges to $ (X^i_t)_{t \in [0,T]} $ in $ \sko{\R} $ almost surely, for $ i \in \{1,2\} $.
	\end{enumerate}
	
	For $ n \geq 1 $, we then define
	\begin{align*}
		L_n(t) := \frac{1}{\sqrt{n}} \int_{0}^{nt} \frac{\1{\xi^{1,n}_s = \xi^{2,n}_s}}{N(\omega, \xi^{1,n}_s)} ds.
	\end{align*}
	We then prove the following, whose statement is identical to Lemma~3.4 in \cite{forien_stepping_2019}, but now with the correct value of $ \gamma $.
	
	\begin{lemma} \label{lemma:convergence_local_time}
		As $ n \to \infty $, for $ \mu $-almost every $ \omega \in \Omega $, $ (L_n(t), t \in [0,T]) $ converges in probability to $ ( \gamma L^0_t(X^1-X^2), t \in [0,T]) $ in the Skorokhod topology, where $ L^0_t(X^1-X^2) $ is the local time at zero of $ X^1-X^2 $ and $ \gamma $ is given by
		\begin{equation} \label{gamma_general}
			\gamma = \int_\Omega \frac{\pi(\omega)^2}{N(\omega)} \mu(d\omega).
		\end{equation}
	\end{lemma}
	
	Note that, plugging (2.3) of \cite{forien_stepping_2019} in \eqref{gamma_general}, we recover \eqref{gamma}.
	We prove Lemma~\ref{lemma:convergence_local_time} in Section~\ref{sec:proof} below.
	
	Note that $ L_n(t) $ can be related to $ L^0_n(t) $, the time that the two random walks spend in the same location up to time $ t $ (see Lemma~\ref{lemma:cvg_average_diagonal} below).
	%	\color{blue}
	%	Before doing so, let us briefly explain the flaw in the argument in \cite{forien_stepping_2019}.
	%	For $ t \geq 0 $ and $ n \geq 1 $, set
	%	\begin{align*}
		%		L^0_n(t) := \int_{0}^{t} \1{\xi^{1,n}_s = \xi^{2,n}_s} ds.
		%	\end{align*}
	%	and
	%	\begin{align*}
		%		Y^n(t) := T^{\xi^{1,n}_{(L^{0}_n)^{-1}(t)}}\omega, \quad t \geq 0,.
		%	\end{align*}
	%	We can then write $ L_n(t) $ as
	%	\begin{align*}
		%		L_n(t) = \frac{1}{\sqrt{n}} L^0_n(nt) \frac{1}{L^0_n(nt)} \int_{0}^{L^0_n(nt)} \frac{1}{N(Y^n(s))} ds.
		%	\end{align*}
	%	By Proposition~3.3 in \cite{forien_stepping_2019} and by the pointwise ergodic theorem, for each $ n \geq 1 $,
	%	\begin{align} \label{cvg_avg_Y}
		%		\frac{1}{t} \int_{0}^{t} \frac{1}{N(Y^n(s))} ds \: \cvgas{t} \: \int_\Omega \frac{1}{N(\omega)} \frac{\pi(\omega)^2}{\mean{\pi^2}} \mu(d\omega) \quad \mu(d\omega)-a.s.
		%	\end{align}
	It was then stated in \cite{forien_stepping_2019} that $ (\frac{1}{\sqrt{n}} L^0_n(nt), t \in [0,T]) $ converges almost surely to $ (L^0_t(X^1-X^2), t \in [0,T]) $.
	This statement is false, however.
	In fact, we shall see that
	\begin{equation*}
		\frac{1}{\sqrt{n}} L^0_n(nt) \cvgas{n} \left(\mean{\pi^2}\right) \, L^0_t(X^1-X^2),
	\end{equation*}
	which yields the statement of Lemma~\ref{lemma:convergence_local_time} with $ \gamma $ now given by \eqref{gamma_general}.
	%	\color{black}
	
	We note that inequality \eqref{ineq_gamma_N} holds in general, since, by the Cauchy-Schwarz inequality,
	\begin{equation*}
		1 = \int_\Omega \pi(\omega) \mu(d\omega) \leq \left( \int_\Omega \frac{\pi(\omega)^2}{N(\omega)} \mu(d\omega) \right)^{1/2} \left( \int_\Omega N(\omega) \mu(d\omega) \right)^{1/2}.
	\end{equation*}
	
	\section{Scaling limit of the intersection local time of two independent random walks} \label{sec:proof}
	
	As a warm up to the proof of Lemma~\ref{lemma:convergence_local_time}, we prove a result on the scaling limit of the local time of a single random walk, which we believe is of independent interest.
	
	\subsection{Local time of a single random walk} \label{subsec:single_rw}
	
	For $ x \in \Z $, define $ (L^n_t(x), t \geq 0) $ as
	\begin{equation*}
		L^n_t(x) := \frac{1}{\sqrt{n}} \int_{0}^{nt} \1{\xi^{1,n}_s = x} ds.
	\end{equation*}
	Moreover, for $ x \in \R $, let $ (L_t(x), t \geq 0) $ denote the local time at $ x $ of $ (X^1_t, t \geq 0) $, such that, for any $ x \in \R $,
	\begin{equation*}
		\lim_{\varepsilon \downarrow 0} \frac{1}{2\varepsilon} \int_{0}^{t} \1{\abs{X^1_s - x} \leq \varepsilon} ds \to L_t(x).
	\end{equation*}
	We then prove the following result.
	
	\begin{theorem} \label{thm:local_time_1_rw}
		For any $ T > 0 $ and $ R > 0 $, for $ \mu $-almost every environment $ \omega \in \Omega $,
		\begin{equation*}
			\sup_{t \in [0,T]} \max_{x \in B(0,R \sqrt{n}) \cap \Z} \Eq[0]{ \, \abs{ \frac{L^n_t(x)}{\pi(\omega, x)} - L_t(x/\sqrt{n}) } \, } \cvgas{n} 0.
		\end{equation*}
	\end{theorem}
	
	To prove this result, let us introduce the notation
	\begin{equation*}
		H_t(x) := \int_{0}^{t} \frac{\1{\xi^{1,n}_s = x}}{\pi(\omega,x)} ds.
	\end{equation*}
	
	\begin{lemma} \label{lemma:H}
		There exists $ C > 0 $ such that, for $ \mu $-almost every environment $ \omega \in \Omega $ and for all $ t \geq 0 $,
		\begin{equation*}
			\sup_{x, y \in \Z} \Eq[x]{ H_t(y) } \leq C \sqrt{t},
		\end{equation*}
		and
		\begin{equation*}
			\sup_{x, y, z \in \Z} \frac{\Eq[x]{(H_t(y) - H_t(z))^2}}{|y-z|^{1/2}} \leq C t^{3/4}.
		\end{equation*}
	\end{lemma}
	
	\begin{proof}
		Recall the definition of $ h^\omega_t(x,y) $ in \eqref{def:ht}.
		The first part follows easily from Lemma~\ref{lemma:bound_h}, writing
		\begin{equation*}
			\Eq[x]{H_t(y)} = \int_{0}^{t} h^\omega_s(x,y) ds \leq \int_{0}^{t} \frac{C}{\sqrt{s}} ds \leq 2C \sqrt{t},
		\end{equation*}
		and noting that the constant $ C $ does not depend either on $ x $ or on $ y $.
		For the second part, we write
		\begin{equation*}
			(H_t(y) - H_t(z))^2 = \int_{0}^{t} \int_{0}^{t} \left( \frac{\1{\xi_s = y}}{\pi(y)} - \frac{\1{\xi_s = z}}{\pi(z)} \right) \left( \frac{\1{\xi_u = y}}{\pi(y)} - \frac{\1{\xi_u = z}}{\pi(z)} \right) du \, ds.
		\end{equation*}
		By symmetry, this is also
		\begin{equation*}
			(H_t(y) - H_t(z))^2 = 2\int_{0}^{t} \int_{s}^{t} \left( \frac{\1{\xi_s = y}}{\pi(y)} - \frac{\1{\xi_s = z}}{\pi(z)} \right) \left( \frac{\1{\xi_u = y}}{\pi(y)} - \frac{\1{\xi_u = z}}{\pi(z)} \right) du \, ds.
		\end{equation*}
		Expanding the product and taking expectations, we obtain by the Markov property,
		\begin{multline} \label{EH2}
			\Eq[x]{(H_t(y) - H_t(z))^2} = 2\int_{0}^{t} \int_{s}^{t} \left\lbrace h^\omega_s(x,y) \left( h^\omega_{u-s}(y,y) - h^\omega_{u-s}(y,z) \right) \right. \\ \left. - h^\omega_s(x,z) \left( h^\omega_{u-s}(z,y) - h^\omega_{u-s}(z,z) \right) \right\rbrace du \, ds.
		\end{multline}
		We can then use Lemma~\ref{lemma:bound_h} and Lemma~\ref{lemma:continuity_h} to obtain, for some constant $ C > 0 $,
		\begin{align*}
			\Eq[x]{(H_t(y) - H_t(z))^2} &\leq C |y-z|^{1/2} \int_{0}^{t} \int_{s}^{t} \frac{1}{\sqrt{s} (u-s)^{3/4}} du \, ds \\
			&= 4 C |y-z|^{1/2} \int_{0}^{t} \frac{(t-s)^{1/4}}{\sqrt{s}} ds \\
			&\leq 8 C t^{3/4} |y-z|^{1/2},
		\end{align*}		
		yielding the second part of the statement.
	\end{proof}
	
	Using Lemma~\ref{lemma:H}, we now prove Theorem~\ref{thm:local_time_1_rw}.
	
	\begin{proof}[Proof of Theorem~\ref{thm:local_time_1_rw}]
		First set
		\begin{equation*}
			H^n_t(x) := \frac{1}{\sqrt{n}} H_{nt}(x) = \frac{L^n_t(x)}{\pi(\omega,x)}.
		\end{equation*}
		Then, inspired by the proof of Theorem~\ref{thm:local_clt} given in \cite{forien_stepping_2019} (itself adapted from \cite{derrien_local_2015}), we write, for $ \delta > 0 $ and $ y \in B(0,R) \cap \frac{1}{\sqrt{n}} \Z $,
		\begin{align*}
			H^n_t(\sqrt{n} y) - L_t(y) &= H^n_t(\sqrt{n} y) \left( 1 - \frac{1}{2 \delta \sqrt{n}} \sum_{z \in B(\sqrt{n} y, \delta \sqrt{n}) \cap \Z} \pi(\omega, z) \right) \\
			&\quad+ \frac{1}{2 \delta \sqrt{n}} \sum_{z \in B(\sqrt{n} y, \delta \sqrt{n}) \cap \Z} \pi(\omega, z) \left( H^n_t(\sqrt{n} y) - H^n_t(z) \right) \\
			&\quad+ \frac{1}{2 \delta n} \int_{0}^{nt} \1{\xi^{1,n}_s \in B(\sqrt{n} y, \delta \sqrt{n})} ds - \frac{1}{2\delta} \int_{0}^{t} \1{X^1_s \in B(y,\delta)} ds \\
			&\quad+ \frac{1}{2\delta} \int_{0}^{t} \1{X^1_s \in B(y, \delta)} ds - L_t(y), \\
			&=: A^n_\delta(t,y) + B^n_\delta(t,y) + C^n_\delta(t,y) + D_\delta(t,y).
		\end{align*}
		where we have used the fact that
		\begin{equation*}
			\frac{1}{2 \delta \sqrt{n}} \sum_{z \in B(\sqrt{n} y, \delta \sqrt{n}) \cap \Z} \pi(\omega, z) H^n_t(z) = \frac{1}{2 \delta n} \int_{0}^{nt} \1{\xi^{1,n}_s \in B(\sqrt{n} y, \delta \sqrt{n})} ds.
		\end{equation*}
		We then bound each term separately.
		
		\paragraph*{Bound on $ A^n_\delta $}
		
		Recalling the notation from \cite{forien_stepping_2019}
		\begin{align} \label{def:Pi}
			\Pi^{n,\delta}(\omega,x) = \frac{1}{2 \delta \sqrt{n}} \sum_{k \in B(\sqrt{n} x, \delta \sqrt{n}) \cap \Z} \pi(\omega,k),
		\end{align}
		By Lemma~\ref{lemma:H},
		\begin{align*}
			\sup_{y \in B(0, R) \cap \frac{1}{\sqrt{n}}\Z} \Eq[0]{ | A^n_\delta(t,y) | } &= \sup_{y \in B(0, R) \cap \frac{1}{\sqrt{n}}\Z} \frac{1}{\sqrt{n}} \Eq[0]{ H_{nt}(\sqrt{n} y) } \abs{ 1 - \Pi^{n,\delta}(\omega,y) } \\
			&\leq C \sqrt{t} \sup_{y \in B(0,R)} \abs{ 1 - \Pi^{n,\delta}(\omega,y) }.
		\end{align*}
		By Lemma~4.3 in \cite{forien_stepping_2019}, the right-hand side converges to 0 as $ n \to \infty $ for $ \mu $-almost every environment $ \omega \in \Omega $, yielding the desired bound on $ A^n_\delta $.
		
		\paragraph*{Bound on $ B^n_\delta $}
		
		By Jensen's inequality,
		\begin{equation*}
			\Eq[0]{| B^n_\delta(t,y) | } \leq \frac{1}{2 \delta \sqrt{n}} \sum_{z \in B(\sqrt{n} y, \delta \sqrt{n}) \cap \Z} \pi(\omega, z) \, \Eq[0]{ \left( H^n_t(\sqrt{n} y) - H^n_t(z) \right)^2}^{1/2}.
		\end{equation*}
		By \eqref{uniform_ellipticity}, $ \pi(\omega, z) \leq K $ and, by Lemma~\ref{lemma:H}, we obtain
		\begin{equation*}
			\sup_{y \in B(0, R) \cap \frac{1}{\sqrt{n}} \Z} \Eq[0]{ | B^n_\delta(t,y) | } \leq C K \frac{1}{\sqrt{n}} (nt)^{3/8} (\delta \sqrt{n})^{1/4} = C K t^{3/8} \delta^{1/4}.
		\end{equation*}
		
		\paragraph*{Bound on $ C^n_\delta $}
		
		By a change of variables,
		\begin{equation*}
			C^n_\delta(t,y) = \frac{1}{2\delta} \int_{0}^{t} \1{\tilde{\xi}^n_s \in B(y, \delta)} ds - \frac{1}{2\delta} \int_{0}^{t} \1{X^1_s \in B(y,\delta)} ds,
		\end{equation*}
		where $ \tilde{\xi}^n_s := \frac{1}{\sqrt{n}} \xi^{1,n}_{ns} $.
		%		\begin{equation} \label{bound_Cn}
			%			\abs{C^n_\delta(t,y)} \leq \frac{1}{2\delta} \int_{0}^{t} \left( \abs{\1{\tilde{\xi}^n_s \leq y + \delta} - \1{X^1_s \leq y + \delta}} + \abs{\1{\tilde{\xi}^n_s \leq y - \delta} - \1{X^1_s \leq y - \delta}} \right) ds.
			%		\end{equation}\color{blue}
		Then, on the event
		\begin{equation*}
			A^n_\varepsilon := \left\lbrace \sup_{s \in [0,T]} | \tilde{\xi}^n_s - X^1_s | \leq \varepsilon \right\rbrace,
		\end{equation*}
		we have
		\begin{equation*}
			\abs{\1{\tilde{\xi}^n_s \in B(y, \delta)} - \1{X^1_s \in B(y,\delta)}} \leq \1{\abs{X^1_s - (y + \delta)} \leq \varepsilon} + \1{\abs{X^1_s - (y-\delta)} \leq \varepsilon}.
		\end{equation*}
		%		for any $ z \in \R $ and $ s \in [0,T] $,
		%		\begin{equation*}
			%			\1{X^1_s \leq z - \varepsilon} \leq \1{\tilde{\xi}^n_s \leq z} \leq \1{X^1_s \leq z + \varepsilon}.
			%		\end{equation*}
		%		Moreover, the same inequality holds trivially with $ \1{X^1_s \leq z} $ in the middle.
		%		As a result,
		%		\begin{align*}
			%			\abs{\1{\tilde{\xi}^n_s \leq z} - \1{X^1_s \leq z}} &\leq \1{X^1_s \leq z + \varepsilon} - \1{X^1_s \leq z - \varepsilon} \\
			%			&= \1{\abs{X^1_s - z} \leq \varepsilon}.
			%		\end{align*}
		%		Plugging this in \eqref{bound_Cn}, we obtain, on $ A^n_\varepsilon $,
		%		\begin{equation*}
			%			\abs{C^n_\delta(t,y)} \leq \frac{1}{2\delta} \int_{0}^{t} \left( \1{\abs{X^1_s - (y + \delta)} \leq \varepsilon} + \1{\abs{X^1_s - (y-\delta)} \leq \varepsilon} \right) ds.
			%		\end{equation*}\color{black}
		We thus obtain
		\begin{multline*}
			\Eq[0]{\abs{C^n_\delta(t,y)}} \leq \frac{1}{2\delta} \int_{0}^{t} \left[ \Pq[0]{\abs{X^1_s - (y+\delta)} \leq \varepsilon} + \Pq[0]{\abs{X^1_s - (y - \delta)} \leq \varepsilon}\right] ds \\ + \frac{t}{2\delta} \Pq[0]{\sup_{s \in [0,T]} | \tilde{\xi}^n_s - X^1_s | > \varepsilon}.
		\end{multline*}
		We then note that
		\begin{equation*}
			\sup_{z \in \R} \int_{0}^{t} \Pq[0]{\abs{X^1_s - z} \leq \varepsilon} ds \leq 2\varepsilon \int_{0}^{t} \frac{ds}{\sqrt{2\pi \sigma^2 s}} = \frac{4 \varepsilon \sqrt{t}}{\sqrt{2 \pi \sigma^2}}.
		\end{equation*}
		Since $ (\tilde{\xi}^n_s, s \in [0,T]) $ converges to $ (X^1_s, s \in [0,T]) $ in probability in $ D([0,T], \R) $ as $ n \to \infty $, and since $ s \mapsto X^1_s $ is almost surely continuous, for any $ \varepsilon > 0 $,
		\begin{equation*}
			\lim_{n \to \infty} \Pq[0]{\sup_{s \in [0,T]} | \tilde{\xi}^n_s - X^1_s | > \varepsilon} = 0.
		\end{equation*}
		Hence
		\begin{equation*}
			\limsup_{n \to \infty} \sup_{t \in [0,T]} \sup_{y \in \R} \Eq[0]{\abs{C^n_\delta(t,y)}} \leq \frac{4 \varepsilon \sqrt{T}}{\delta \sqrt{2 \pi \sigma^2}}.
		\end{equation*}
		Letting $ \varepsilon $ tend to zero, we obtain
		\begin{equation*}
			\lim_{n \to \infty} \sup_{t \in [0,T]} \sup_{y \in \R} \Eq[0]{\abs{C^n_\delta(t,y)}} = 0.
		\end{equation*}
		
		\paragraph*{Bound on $ D_\delta $}
		
		By the occupation density formula for Brownian motion \citep[Corollary~VI.1.6]{revuz_continuous_2013},
		\begin{equation*}
			D_\delta(t,y) = \frac{1}{2\delta} \int_{y - \delta}^{y + \delta} (L_t(z) - L_t(y)) dz.
		\end{equation*}
		By the local uniform continuity of local time \citep[Corollary~VI.1.8]{revuz_continuous_2013},
		\begin{equation*}
			\sup_{t \in [0,T]} \sup_{y \in B(0,R)} \abs{D_\delta(t,y)} \to 0
		\end{equation*}
		almost surely as $ \delta \to 0 $.
		By the dominated convergence theorem, we then obtain
		\begin{equation*}
			\lim_{\delta \to 0} \sup_{t \in [0,T]} \sup_{y \in B(0,R)} \Eq[0]{\abs{D_\delta(t,y)}} = 0.
		\end{equation*}
		
		\paragraph*{Conclusion}
		
		Combining all the above bounds, we obtain that, for any $ T > 0 $, there exists a constant $ C_T > 0 $ such that, for $ \mu $-almost every environment $ \omega \in \Omega $ and for any $ \delta > 0 $,
		\begin{multline*}
			\limsup_{n \to \infty} \sup_{t \in [0,T]} \max_{y \in B(0,R \sqrt{n}) \cap \Z} \Eq[0]{ \, \abs{ \frac{L^n_t(y)}{\pi(\omega, y)} - L_t(y/\sqrt{n}) } \, } \\ \leq C_T \delta^{1/4} + \sup_{t \in [0,T]} \sup_{y \in B(0,R)} \Eq[0]{\abs{D_\delta(t,y)}}.
		\end{multline*}
		Letting $ \delta $ tend to zero, we obtain the statement of Theorem~\ref{thm:local_time_1_rw}.
	\end{proof}
	
	\subsection{Intersection local time of two random walks} \label{subsec:proof_2_rw}
	
	We now prove Lemma~\ref{lemma:convergence_local_time}, adapting the ideas of the proof of Theorem~\ref{thm:local_time_1_rw} to the setting of two independent random walks.
	We start by setting, for $ x, y \in \Z $,
	\begin{equation*}
		H_t(x, y) := \int_{0}^{t} \frac{\1{\xi^{1,n}_s = x, \xi^{2,n}_s = y}}{\pi(\omega,x) \pi(\omega,y)} ds.
	\end{equation*}
	We then prove the following.
	
	\begin{lemma} \label{lemma:continuity_H_xy}
		For any measurable bounded function $ f : \Omega^3 \to \R $, there exists a constant $ C > 0 $ such that, for $ \mu $-almost every environment $ \omega \in \Omega $, any bounded set $ A \subset \R^2 $ and for all $ t \geq 0 $,
		\begin{multline*}
			\Eq{\left( \sum_{x \in \Z}  \frac{1}{| A \cap \Z^2|} \sum_{(y,z) \in A \cap \Z^2} f(T^x \omega, T^{x+y}\omega, T^{x+z} \omega) (H_t(x+y,x+z) - H_t(x,x)) \right)^2} \\ \leq C t^{3/4} \sup_{(y, z) \in A} (|y|^{1/2} + |z|^{1/2}).
		\end{multline*}
	\end{lemma}
	
	\begin{proof}
		In the following, we write
		\begin{equation*}
			f(\omega, x, y, z) = f(T^x \omega, T^{x+y}\omega, T^{x+z} \omega).
		\end{equation*}
		We start by using Fubini's theorem, using the fact that, for any $ (y,z) \in A \cap \Z^2 \cup \lbrace (0,0) \rbrace $,
		\begin{equation*}
			\sum_{x \in \Z} H_t(x+y,x+z) \leq K^2 \int_{0}^{t} \sum_{x \in \Z} \1{\xi^{1,n}_s = x + y} ds = t K^2,
		\end{equation*}
		almost surely, to write
		\begin{multline} \label{Ef_HH}
			\Eq{\left( \sum_{x \in \Z} \frac{1}{| A \cap \Z^2|} \sum_{(y,z) \in A \cap \Z^2} f(\omega, x, y, z) (H_t(x+y,x+z) - H_t(x,x)) \right)^2} \\ = \frac{1}{| A \cap \Z^2 |^2} \sum_{x, x' \in \Z} \sum_{\underset{(y',z') \in A \cap \Z^2}{(y,z) \in A \cap \Z^2}} f(\omega, x, y, z) f(\omega, x', y', z') \\ \times  \Eq{ (H_t(x+y,x+z) - H_t(x,x)) (H_t(x'+y',x'+z') - H_t(x',x')) }.
		\end{multline}
		The expectation on the right-hand side can be computed exactly as in \eqref{EH2}, using the fact that, for $ 0 \leq s \leq u \leq t $,
		\begin{equation*}
			\Eq{ \frac{\1{\xi^{1,n}_s = x, \xi^{2,n}_s = y} \1{\xi^{1,n}_u = x', \xi^{2,n}_u = y'}}{\pi(\omega,x)\pi(\omega,y)\pi(\omega,x')\pi(\omega,y')} } = h^\omega_s(\xi^{1,n}_0, x) h^\omega_s(\xi^{2,n}_0, y) h^\omega_{u-s}(x,x') h^\omega_{u-s}(y,y').
		\end{equation*}
		We then obtain that the expectation on the right-hand side of \eqref{Ef_HH} can be written as the sum of four terms of the form
		\begin{equation*}
			2 \!\! \int_{0}^{t} \!\!\! \int_{s}^{t} \!\!\! h^\omega_s(\xi^{i,n}_0\!,x+w) h^\omega_s(\xi^{j,n}_0\!,x+v) h^\omega_{u-s}(x+w,x'+w') (h^\omega_{u-s}(x+v,x'+v') - h^\omega_{u-s}(x+v,x')) du ds,
		\end{equation*}
		where either $ (w,v) \in A $ or $ (v,w) \in A $ or $ (v,w) = (0,0) $ and either $ (w',v') \in A $ or $ (w',v') = (0,y) $ for some $ (y,z) \in A $.
		%		We can again use a symmetry argument to obtain that the expectation in the above sum can be replaced by
		%		\begin{multline*}
			%			2 \int_{0}^{t} \int_{s}^{t} \left\lbrace h^\omega_s(\xi^{1,n}_0, x+y) h^\omega_s(\xi^{2,n}_0, x+z) \left( h^\omega_{u-s}(x+y,x'+y') h^\omega_{u-s}(x+z,x'+z') \right.\right. \\ \left. \left. - h^\omega_{u-s}(x+y,x') h^\omega_{u-s}(x+z,x') \right) \right.\\ \left. - h^\omega_s(\xi^{1,n}_0, x) h^\omega_s(\xi^{2,n}_0, x)) \left( h^\omega_{u-s}(x, x'+y') h^\omega_{u-s}(x,x'+z') - h^\omega_{u-s}(x,x') h^\omega_{u-s}(x,x') \right) \right\rbrace du ds.
			%		\end{multline*}
		By Lemma~\ref{lemma:bound_h} and Lemma~\ref{lemma:continuity_h}, there exists a constant $ C > 0 $ such that, for $ \mu $-almost every $ \omega \in \Omega $, this is bounded from above by
		\begin{equation*}
			2 C \int_{0}^{t} \int_{s}^{t} h^\omega_s(\xi^{i,n}_0,x+w) \frac{1}{\sqrt{s}} h^\omega_{u-s}(x+w,x'+w') \frac{|v'|^{1/2}}{(u-s)^{3/4}} du ds.
			%			&\left.+ h^\omega_s(\xi^{1,n}_0, x+y) \frac{1}{\sqrt{s}} h^\omega_{u-s}(x+z,x') \frac{|y'|^{1/2}}{(u-s)^{3/4}} \right. \\ 
			%			&\left.+ h^\omega_s(\xi^{1,n}_0, x) \frac{1}{\sqrt{s}} h^\omega_{u-s}(x,x'+y') \frac{|z'|^{1/2}}{(u-s)^{3/4}} \right. \\ 
			%			&\left.+ h^\omega_s(\xi^{1,n}_0, x) \frac{1}{\sqrt{s}} h^\omega_{u-s}(x,x') \frac{|y'|^{1/2}}{(u-s)^{3/4}} \right\rbrace du ds,
		\end{equation*}
		%		for $ \mu $-almost every $ \omega \in \Omega $, for some constant $ C > 0 $.
		We can then use $ |v'|^{1/2} \leq \sup_{(y,z) \in A} (|y|^{1/2}+|z|^{1/2}) $ and the fact that, by \eqref{uniform_ellipticity},
		\begin{equation} \label{sum_h}
			\sum_{x \in \Z} h^\omega_s(x_0,x) \leq K \sum_{x \in \Z} \Pq[x_0]{\xi_s = x} = K,
		\end{equation}
		to obtain that there exists a constant $ C > 0 $ such that, 
		\begin{multline*}
			\Eq{\left( \sum_{x \in \Z} f(T^x \omega) \frac{1}{| A \cap \Z^2|} \sum_{(y,z) \in A \cap \Z^2} (H_t(x+y,x+z) - H_t(x,x)) \right)^2} \\ \leq C \left( \sup_{(y,z) \in A} (|y|^{1/2} + |z|^{1/2}) \right) \int_{0}^{t} \frac{1}{\sqrt{s}} \int_{s}^{t} \frac{1}{(u-s)^{3/4}} du ds,
		\end{multline*}
		%		We then compute
		%		\begin{align*}
			%			\int_{0}^{t} \frac{1}{\sqrt{s}} \int_{s}^{t} \frac{1}{(u-s)^{3/4}} du ds &= 4 \int_{0}^{t} \frac{(t-s)^{1/4}}{\sqrt{s}} ds \\
			%			&\leq 8 t^{3/4},
			%		\end{align*}
		%		which concludes the proof of the lemma.\color{black}
		from which the statement easily follows.
	\end{proof}
	
	With the help of Lemma~\ref{lemma:continuity_H_xy} we prove the following.
	Recall that $ L^0_t(X^1-X^2) $ denotes the intersection local time of $ X^1 $ and $ X^2 $, i.e. the local time at 0 of $ X^1 - X^2 $.
	
	\begin{theorem} \label{thm:intersection_local_time}
		For any $ T > 0 $, for $ \mu $-almost every environment $ \omega \in \Omega $,
		\begin{equation*}
			\sup_{t \in [0,T]} \Eq{ \, \abs{ \frac{1}{\sqrt{n}} \int_{0}^{nt} \frac{\1{\xi^{1,n}_s = \xi^{2,n}_s}}{\pi(\xi^{1,n}_s)^2} ds - L^0_t(X^1-X^2) } \, } \cvgas{n} 0.
		\end{equation*}
	\end{theorem}
	
	Compare this with what was wrongly stated in \cite{forien_stepping_2019}, namely that
	\begin{equation*}
		\frac{1}{\sqrt{n}} \int_{0}^{nt} \1{\xi^{1,n}_s = \xi^{2,n}_s} ds \to L^0_t(X^1-X^2).
	\end{equation*}
	
	\begin{proof}
		First set
		\begin{equation*}
			H^n_t(x,y) := \frac{1}{\sqrt{n}} H_{nt}(x,y).
		\end{equation*}
		Then
		\begin{equation*}
			\frac{1}{\sqrt{n}} \int_{0}^{nt} \frac{\1{\xi^{1,n}_s = \xi^{2,n}_s}}{\pi(\xi^{1,n}_s)^2} ds = \sum_{x \in \Z} H^n_t(x,x).
		\end{equation*}
		We then proceed similarly to the proof of Theorem~\ref{thm:local_time_1_rw}.
		For $ z_1, z_2 \in \Z $ and $ \delta > 0 $, set
		\begin{equation*}
			V^n_\delta(z_1, z_2) := \frac{1}{4 \delta^2 \sqrt{n}} \sum_{x \in \Z} \1{|x - z_1| \leq \delta \sqrt{n}, |x-z_2| \leq \delta \sqrt{n}}.
		\end{equation*}
		We then write, for $ \delta > 0 $,
		\begin{multline*}
			\sum_{x \in \Z} H^n_t(x,x) - L^0_t(X^1 - X^2) \\
			\begin{aligned}
				&= \sum_{x \in \Z} H^n_t(x,x) \left( 1 - \frac{1}{(2\delta \sqrt{n})^2} \sum_{y_1, y_2 \in B(0,\delta \sqrt{n}) \cap \Z} \pi(x + y_1) \pi(x + y_2) \right) \\
				&\quad + \sum_{x \in \Z} \frac{1}{(2\delta \sqrt{n})^2} \sum_{y_1, y_2 \in B(0,\delta \sqrt{n}) \cap \Z}  \pi(x + y_1) \pi(x + y_2) \left( H^n_t(x,x) - H^n_t(x+y_1,x+y_2) \right) \\
				&\quad + \frac{1}{n} \int_{0}^{nt} V^n_\delta(\xi^{1,n}_s, \xi^{2,n}_s) ds - \int_{0}^{t} \frac{1}{4 \delta^2} (2\delta - |X^1_s - X^2_s|)^+ ds \\
				&\quad + \int_{0}^{t} \frac{1}{4 \delta^2} (2\delta - |X^1_s - X^2_s|)^+ ds - L^0_t(X^1 - X^2) \\
				&=: A^n_\delta(t) + B^n_\delta(t) + C^n_\delta(t) + D_\delta(t),
			\end{aligned}
		\end{multline*}
		where we have used the fact that
		\begin{equation*}
			\sum_{x \in \Z} \frac{1}{(2\delta \sqrt{n})^2} \sum_{y_1, y_2 \in B(0,\delta \sqrt{n}) \cap \Z}  \pi(x + y_1) \pi(x + y_2) H^n_t(x+y_1,x+y_2) = \frac{1}{n} \int_{0}^{nt} V^n_\delta(\xi^{1,n}_s, \xi^{2,n}_s) ds.
		\end{equation*}
		We then bound each term separately, adapting the arguments of the proof of Theorem~\ref{thm:local_time_1_rw}.
		
		\paragraph*{Bound on $ A^n_\delta $}
		
		Recalling \eqref{def:Pi},
		\begin{equation*}
			A^n_\delta(t) = \sum_{x \in \frac{1}{\sqrt{n}} \Z} H^n_t(\sqrt{n} x, \sqrt{n} x) \left( 1 - \Pi^{n,\delta}(\omega, x)^2 \right).
		\end{equation*}
		Then, for any compact set $ \mathcal{C} \subset \R $, using \eqref{uniform_ellipticity} in the second term on the right,
		\begin{equation} \label{bound_A}
			|A^n_\delta(t)| \leq \sup_{x \in \mathcal{C}} | 1-\Pi^{n,\delta}(\omega,x)^2 | \sum_{x \in \Z} H^n_t(x, x) + (1+K^2) \sum_{x \in \frac{1}{\sqrt{n}} \Z \cap \mathcal{C}^c} H^n_t(\sqrt{n} x, \sqrt{n} x).
		\end{equation}
		Combining Lemma~\ref{lemma:bound_h} and the argument used in \eqref{sum_h}, we obtain, for $ A \subset \R $ and all $ t \geq 0 $,
		%		For a set $ A \subset \R $, by Lemma~\ref{lemma:bound_h}, we obtain
		%		\begin{equation*}
			%			\Eq{ \sum_{x \in \frac{1}{\sqrt{n}} \Z \cap A} H^n_t(\sqrt{n}x, \sqrt{n}x)} \leq \sum_{x \in \frac{1}{\sqrt{n}} \Z \cap A} \frac{1}{\sqrt{n}} \int_{0}^{nt} \frac{C}{\sqrt{s}} h^\omega_s(\xi^{2,n}_0, \sqrt{n} x) ds.
			%		\end{equation*}
		%		By the same argument as in \eqref{sum_h},
		%		\begin{equation*}
			%			\sum_{x \in \frac{1}{\sqrt{n}} \Z \cap A} h^\omega_s(\xi^{2,n}_0, \sqrt{n} x) \leq K \Pq[\xi^{2,n}_0]{\xi_s \in \sqrt{n} A}.
			%		\end{equation*}
		%		By a change of variables, we obtain
		\begin{equation*}
			\Eq{ \sum_{x \in \frac{1}{\sqrt{n}} \Z \cap A} H^n_t(\sqrt{n}x, \sqrt{n}x)} \leq C \int_{0}^{t} \frac{1}{\sqrt{s}} \Pq{\tilde{\xi}^{2,n}_s \in A} ds,
		\end{equation*}
		where $ \tilde{\xi}^{n,i}_s := \frac{1}{\sqrt{n}} \xi^{n,i}_{ns} $.
		%		Coming back to \eqref{bound_A}, we obtain, using \eqref{uniform_ellipticity} to bound $ \sup_{x \in \R}  | 1 - \Pi^{n,\delta}(\omega,x) | $,
		%		\begin{equation*}
			%			\Eq[\xi^{1,n}_0,\xi^{2,n}_0]{| A^n_\delta(t)|} \leq C \sqrt{t} \left( \sup_{x \in \mathcal{C}} | 1 - \Pi^{n,\delta}(\omega,x) | + \sup_{s \in [0,t]} \Pq{\tilde{\xi}^{2,n}_s \notin \mathcal{C}} \right).
			%		\end{equation*}
		Moreover, by Lemma~4.3 in \cite{forien_stepping_2019}, the first term on the right of \eqref{bound_A} vanishes as $ n \to \infty $ for any compact $ \mathcal{C} $ and, by the tightness of $ (\tilde{\xi}^{2,n}, n \geq 1) $, for any $ \varepsilon > 0 $, there exists a compact set $ \mathcal{C} $ such that
		\begin{equation*}
			\sup_{n \geq 1} \sup_{s \in [0,t]} \Pq{\tilde{\xi}^{2,n}_s \notin \mathcal{C}} \leq \varepsilon.
		\end{equation*}
		This shows that
		\begin{equation*}
			\lim_{n \to \infty} \Eq{| A^n_\delta(t)|} = 0.
		\end{equation*}
		
		\paragraph*{Bound on $ B^n_\delta $}
		
		For the bound on $ B^n_\delta $, we apply Lemma~\ref{lemma:continuity_H_xy} with $ A = B(0,\delta \sqrt{n})^2 $ and $ f(\omega, x, y, z) = \pi(\omega,x + y) \pi(\omega,x + z) $ to obtain
		\begin{equation*}
			\Eq{ | B^n_\delta(t) |^2 } \leq \frac{C}{n} (nt)^{3/4} (\delta \sqrt{n})^{1/2} = C t^{3/4} \delta^{1/2}.
		\end{equation*}
		
		\paragraph*{Bound on $ C^n_\delta $}
		
		By a change of variables,
		\begin{equation*}
			\frac{1}{n} \int_{0}^{nt} V^n_\delta(\xi^{1,n}_s, \xi^{2,n}_s) ds = \int_{0}^{t} V^n_\delta(\sqrt{n} \tilde{\xi}^{n,1}_s, \sqrt{n} \tilde{\xi}^{n,2}_s) ds.
		\end{equation*}
		In addition, we note that, as $ n \to \infty $, $ V^n_\delta(\sqrt{n} x_1, \sqrt{n} x_2) $ converges to $ V_\delta(x_1,x_2) $ uniformly over $ x_1, x_2 \in \R $, where
		%		We then note that, for $ z_1, z_2 \in \R $ and $ \delta > 0 $,
		%		\begin{equation*}
			%			B(z_1,\delta) \cap B(z_2, \delta) = B\left( \frac{z_1+z_2}{2}, \left( \delta - \frac{|z_1-z_2|}{2} \right)^+ \right).
			%		\end{equation*}
		%		Then, for $ x_1, x_2 \in \R $,
		%		\begin{equation*}
			%			V^n_\delta(\sqrt{n} x_1, \sqrt{n} x_2) = \frac{1}{4\delta^2 \sqrt{n}} \abs{ B\left( \sqrt{n} \frac{x_1 + x_2}{2}, \sqrt{n} \left( \delta - \frac{|x_1-x_2|}{2} \right)^+ \right) \cap \Z }.
			%		\end{equation*}
		%		We then conclude that this quantity converges uniformly over $ x_1, x_2 \in \R $ to
		\begin{equation*}
			V_\delta(x_1, x_2) := \frac{1}{4\delta^2} \left( 2 \delta - |x_1-x_2| \right)^+.
		\end{equation*}
		Combined with the fact that $ (\tilde{\xi}^{n,i}_t, t \in [0,T]) $ converges in probability to  $ (X^i_t, t \in [0,T]) $, this shows that $ C^n_\delta(t) \to 0 $ as $ n \to \infty $ in probability.
		Since $ C^n_\delta(t) $ is also bounded from above by a deterministic constant (also independent of $ t \in [0,T] $), this shows that
		\begin{equation*}
			\sup_{t \in [0,T]} \Eq{| C^n_\delta(t) |} \to 0
		\end{equation*}
		as $ n \to \infty $ for $ \mu $-almost every $ \omega \in \Omega $.
		
		\paragraph*{Bound on $ D_\delta $}
		
		For the last term, by similar arguments as in Section~\ref{subsec:single_rw}, we have
		%		 we write
		%		\begin{equation*}
			%			\int_{0}^{t} \frac{1}{4 \delta^2} (2\delta - |X^1_s - X^2_s|)^+ ds = \int_{-2\delta}^{2\delta} \frac{1}{4 \delta^2} (2\delta - |x|)^+ L^x_t(X^1-X^2) dx.
			%		\end{equation*}
		%		We then note that
		%		\begin{equation*}
			%			\int_{-2\delta}^{2\delta} \frac{1}{4 \delta^2} (2\delta - |x|)^+ dx = 2 \int_{0}^{2 \delta} \frac{1}{4 \delta^2} (2\delta - x) dx = 1.
			%		\end{equation*}
		%		It follows that
		\begin{equation*}
			D_\delta(t) = \int_{-2\delta}^{2\delta} \frac{1}{4 \delta^2} (2\delta - |x|)^+ \left[L^x_t(X^1-X^2) - L^0_t(X^1-X^2)\right] dx.
		\end{equation*}
		We can then conclude again, using Corollary~VI.1.8 of \cite{revuz_continuous_2013}, that
		\begin{equation*}
			\lim_{\delta \to 0} \sup_{t \in [0,T]} \Eq{|D_\delta(t)|} = 0.
		\end{equation*}
		
		\paragraph*{Conclusion}
		
		Combining the above bounds, we obtain that there exists a constant $ C > 0 $ such that, for any $ \delta > 0 $, $ t \in [0,T] $ and $ \mu $-almost every $ \omega \in \Omega $,
		\begin{equation*}
			\limsup_{n \to \infty} \sup_{t \in [0,T]} \Eq{ \, \abs{ \sum_{x \in \Z} H^n_t(x,x) - L^0_t(X^1-X^2) } \, } \leq C \delta^{1/4} + \sup_{t \in [0,T]} \Eq{|D_\delta(t)|}.
		\end{equation*}
		Letting $ \delta $ tend to zero, we then obtain the result.
	\end{proof}
	
	Finally, we can now prove Lemma~\ref{lemma:convergence_local_time}.
	We can in fact prove a more general result from which Lemma~\ref{lemma:convergence_local_time} directly follows.
	
	\begin{lemma} \label{lemma:cvg_average_diagonal}
		For any bounded and measurable $ f : \Omega \to \R $, for any $ T > 0 $ and for $ \mu $-almost every environment $ \omega \in \Omega $,
		\begin{equation*}
			\sup_{t \in [0,T]} \Eq{ \, \abs{ \frac{1}{\sqrt{n}} \int_{0}^{nt} f(T^{\xi^{1,n}_s} \omega) \1{\xi^{1,n}_s = \xi^{2,n}_s} ds - \left( \mean{f \pi^2} \right) L^0_t(X^1-X^2) } \, } \cvgas{n} 0.
		\end{equation*}
	\end{lemma}
	
	Lemma~\ref{lemma:convergence_local_time} then follows with $ f(\omega) = \frac{1}{N(\omega)} $.
	
	\begin{proof}
		We start by writing
		\begin{equation*}
			\frac{1}{\sqrt{n}} \int_{0}^{nt} f(T^{\xi^{1,n}_s} \omega) \1{\xi^{1,n}_s = \xi^{2,n}_s} ds = \sum_{x \in \Z} H^n_t(x,x) f(x) \pi(x)^2.
		\end{equation*}
		It follows that, for any $ \delta > 0 $,
		\begin{multline*}
			\frac{1}{\sqrt{n}} \int_{0}^{nt} f(T^{\xi^{1,n}_s} \omega) \1{\xi^{1,n}_s = \xi^{2,n}_s} ds - \left( \mean{f \pi^2} \right) L^0_t(X^1-X^2) \\
			\begin{aligned}
				&= \sum_{x \in \Z} \frac{1}{2\delta \sqrt{n}} \sum_{y \in B(x, \delta \sqrt{n}) \cap \Z} f(x) \pi(x)^2 \left( H^n_t(x,x) - H^n_t(y,y) \right) \\
				&\quad + \sum_{y \in \Z} \left( \frac{1}{2\delta \sqrt{n}} \sum_{x \in B(y,\delta \sqrt{n}) \cap \Z} f(x)\pi(x)^2 - \mean{f \pi^2} \right) H^n_t(y,y) \\
				&\quad + \mean{f \pi^2} \left[ \sum_{x \in \Z} H^n_t(x,x) - L^0_t(X^1 - X^2) \right].
			\end{aligned}
		\end{multline*}
		The first term can then be bounded using Lemma~\ref{lemma:continuity_H_xy}, the second term can be bounded as $ A^n_\delta $ in the proof of Theorem~\ref{thm:intersection_local_time}, and the last term tends to zero by Theorem~\ref{thm:intersection_local_time}.
		This concludes the proof of Lemma~\ref{lemma:cvg_average_diagonal}.
	\end{proof}
	
	\bibliography{biblioErratum}
\end{document}